\documentclass{amsart}
\usepackage[latin1]{inputenc}
\usepackage{amsfonts}
\usepackage{amsmath}
\usepackage{amsthm}
\usepackage{amssymb}
\usepackage{latexsym}
\usepackage{enumerate}
\usepackage[all]{xy}
\usepackage{color}

\newtheorem{theorem}{Theorem}[section]
\newtheorem{lemma}[theorem]{Lemma}
\newtheorem{proposition}[theorem]{Proposition}
\newtheorem{corollary}[theorem]{Corollary}

\newtheorem{definition}[theorem]{Definition}

\numberwithin{equation}{section}

\begin{document}
\newcommand{\cc}{\mathfrak{c}}
\newcommand{\N}{\mathbb{N}}
\newcommand{\Q}{\mathbb{Q}}
\newcommand{\C}{\mathbb{C}}
\newcommand{\Z}{\mathbb{Z}}
\newcommand{\R}{\mathbb{R}}
\newcommand{\T}{\mathbb{T}}
\newcommand{\I}{\mathcal{I}}
\newcommand{\J}{\mathcal{J}}
\newcommand{\A}{\mathcal{A}}
\newcommand{\HH}{\mathcal{H}}
\newcommand{\K}{\mathcal{K}}
\newcommand{\B}{\mathcal{B}}
\newcommand{\st}{*}
\newcommand{\PP}{\mathbb{P}}
\newcommand{\SSS}{\mathbb{S}}
\newcommand{\forces}{\Vdash}
\newcommand{\dom}{\text{dom}}
\newcommand{\osc}{\text{osc}}
\newcommand\encircle[1]{%
  \tikz[baseline=(X.base)] 
    \node (X) [draw, shape=circle, inner sep=0] {\strut #1};}

\title[An  extension]{An  extension of compact operators
by compact operators with no nontrivial multipliers}

\author{Saeed Ghasemi}
\address{Institute of Mathematics, Polish Academy of Sciences,
ul. \'Sniadeckich 8,  00-656 Warszawa, Poland $\&$ Institute of Mathematics of the Czech Academy of Sciences, \v Zitn\'a
25, 115 67 Praha 1, Czech Republic }
\email{\texttt{ghasemi@math.cas.cz}}

\author{Piotr Koszmider}
\address{Institute of Mathematics, Polish Academy of Sciences,
ul. \'Sniadeckich 8,  00-656 Warszawa, Poland}
\email{\texttt{piotr.koszmider@impan.pl}}
\thanks{The research of the second named author was partially supported by   grant
PVE Ci\^encia sem Fronteiras - CNPq (406239/2013-4).}

\begin{abstract} We construct a nonhomogeneous, separably represented, type I and approximately finite dimensional
 $C^*$-algebra such
that its multiplier algebra is equal to its unitization. This algebra is an essential  extension
of the algebra $\mathcal K(\ell_2({\mathfrak{c}}))$ of compact operators on a nonseparable Hilbert space by
the algebra  $\mathcal K(\ell_2)$ of compact operators on a separable Hilbert space, where ${\mathfrak{c}}$
 denotes the cardinality
of continuum.  Although
both  $\mathcal K(\ell_2({\mathfrak{c}}))$ and $\mathcal K(\ell_2)$ are stable, our algebra is not.
This sheds light on the
permanence properties of the stability in the nonseparable setting. Namely, unlike in the separable case, 
an extension of a stable nonseparable 
$C^*$-algebra   by $\mathcal K(\ell_2)$ does not have to be stable.
 Our construction can be 
considered as a noncommutative version of Mr\'owka's $\Psi$-space; a space whose one point compactification is
equal to its \v Cech-Stone compactification and is induced by a special uncountable family of
almost disjoint subsets of ${\mathbb{N}}$.  

\end{abstract}

\maketitle

\section{introduction}

Perhaps the simplest example of a locally compact space whose one-point compactification
is equal to the \v Cech-Stone compactification is the first uncountable ordinal $\omega_1$ with the order topology.
This follows from the well-known fact that every real or complex valued continuous function on $\omega_1$
is eventually constant. Another example of such spaces  is  $K\setminus\{x\}$, where  $K$ is a compact extremally disconnected space
and $x$ is a nonisolated point  (Exercise 1H of \cite{gj}).
A noncommutative version of this fact was proved in \cite{sakai} in the context of II$_1$ factors.
In \cite{mrowka} 
Mr\'owka  constructed  a locally compact space  with the same property that  the one-point compactification
and \v Cech-Stone compactification coincide which moreover has the simplest nontrivial Cantor-Bendixson
decomposition, i.e., after removing  a countable dense subset of isolated points we are
left with an uncountable discrete space. In other words, it is  a  separable scattered space 
 of Cantor-Bendixson height $2$ (see
6.4. of \cite{hrusak}).  Such spaces are induced
by uncountable almost disjoint families of infinite subsets of $\N$ (every two distinct members of the family
have finite intersection).
On the level of Banach spaces
of continuous functions or commutative $C^*$-algebras Mr\'owka's
space $X$ satisfies the following short exact sequence 
$$0\rightarrow c_0\xrightarrow{\iota}  C_0(X) \rightarrow
c_0(\cc)\rightarrow 0,$$
where $\iota[c_0]$ is an essential ideal $C_0(X)$, i.e., $\N$ is a dense open subset of $X$.
Here $\cc$ denotes the cardinality of the continuum. In other words, $C_0(X)$ is an essential extension of 
$c_{0}(\cc)$ by $c_0$ (see II.8.4 of \cite{Blackadar}).

In this paper we  produce a noncommutative version of this phenomenon.
It is widely accepted that the noncommutative version of the ideal of finite subsets of $\N$,
or the commutative $C^*$-algebra $c_0$, is the $C^*$-algebra of all compact operators
on a separable Hilbert space. 
The same analogy exists for finite subsets of $\cc$ and the $C^*$-algebra of
all compact operators on the Hilbert space $\ell_2(\cc)$ of density $\cc$.
The roles of the one point compactification and the \v Cech-Stone compactification
of a locally compact, noncompact space $X$
are played by the unitization of a nonunital $C^*$-algebra $\mathcal A$ (which will be denoted by $\tilde{A}$)
 and the
multiplier algebra $\mathcal M(\A)$ of $\A$, respectively.
Thus we are interested in
 an  essential extension of the algebra of compact
operators $\mathcal K(\ell_2(\cc))$
by $\mathcal K(\ell_2)$,  i.e., a $C^*$-algebra $\mathcal A$
satisfying the short exact sequence
$$0\rightarrow \mathcal K(\ell_2)\xrightarrow{\iota} \mathcal A 
\rightarrow\mathcal K(\ell_2(\cc))\rightarrow 0,\leqno(*)$$
where $\iota[\mathcal K(\ell_2)]$ is an essential ideal of
$\mathcal A$.  In 
the main theorem of this paper, Theorem \ref{main}, we construct such an  algebra  $\A$ with
 the required additional property that the multiplier algebra 
$\mathcal M(\mathcal A)$ of $\mathcal A$
is *-isomorphic to the unitization of $\mathcal A$. 
In other words, the corona algebra $\mathcal M(\A)/\A$ of our $\A$ is  *-isomorphic to $\C$. In fact $\A$ has the property that the space  $\mathcal Q \mathcal M(\A)$ of all quasi-multipliers of $\A$ coincides with $\mathcal M(\A)$ and therefore $\mathcal Q \mathcal M(\A)/\A$ is also *-isomorphic to $\C$.
The algebra $\A$ of Theorem \ref{main} is a nonseparable subalgebra of $\B(\ell_2)$, which is type I and approximately finite dimensional
 in the sense that any finite 
subset can be approximated from a finite dimensional subalgebra. 
Moreover $\A$ is a scattered $C^*$-algebra (see \cite{scattered}), which means all of 
its subalgebras are also approximately finite dimensional (\cite{kusuda-af}). Note that the various 
equivalent definitions of approximately finite dimensional
 $C^*$-algebras which are equivalent
in the separable case are no longer equivalent in the nonseparable context (see \cite{ilijas-af}
where a different terminology is used).

For $C^*$-algebras $\B$ and $\mathcal C$, an extension of $\B$ by $\mathcal C$ is a short exact sequence of $C^*$-algebras
$$0\rightarrow \mathcal C\rightarrow \mathcal A 
\rightarrow\mathcal B\rightarrow 0.$$ 
The goal of the extension theory is, given $\B$ and $\mathcal C$, to classify all the extensions of $\B$ by $\mathcal C$ up to 
a suitable equivalence relation. The set of all equivalence classes of extensions of $\B$ by $\mathcal C$ can be equipped with a proper addition which turns it into an abelian semigroup, usually denoted by $Ext(\B,\mathcal C)$, or simply $Ext(\B)$ if $\mathcal C = \K(\ell_2)$. The reader may refer to \cite{blackadar-k} for the details and various definitions regarding extensions of $C^*$-algebras, however in this paper we are not concerned about the structure of $Ext$ semigroups,
although we hope that the extension we construct is a contribution to a more general and future project of  understanding 
the semigroup $Ext(\mathcal K(\ell_2({{\kappa}}))$ for $\omega_1\leq \kappa\leq \cc$.
The ``extension questions" for $C^*$-algebras ask  whether the $C^*$-algebra $\A$ in the
 extension $$0\rightarrow \mathcal C\rightarrow \mathcal A 
\rightarrow\mathcal B\rightarrow 0,$$
satisfies property $\mathcal P$, given that both $\B$ and $\mathcal C$ satisfy $\mathcal P$.
One of the features of 
our extension  is that  $\B$ and $\mathcal C$ are as simple as possible (besides $\B$ being nonseparable), while
$\A$ is quite pathological, which makes it interesting for the questions of this sort. 

In particular if $\mathcal P$ is the stability property of a
$C^*$-algebra (recall that a $C^*$-algebra  is stable if
it is isomorphic to its tensor product by $\K(\ell_2)$) then the above question is usually called 
``the extension question for stable $C^*$-algebras" (see \cite{rordam-stable}). If
 $\mathcal C=\K(\ell_2)$ and $\B$ is a separable $C^*$-algebra, then 
  $\A$ is stable if and only if $\B$ is stable (see Proposition 6.12 of \cite{rordam-stable}; this is  essentially
 a result 
of BDF-theory (\cite{BDF})). In fact a result of Blackadar  (\cite{blackadar-af}) shows that
this holds also if  $\B$ and $\mathcal C$ in the above short exact sequence
 are any separable AF-algebras. This result is extended to extensions of  more general separable $C^*$-algebras in 
 \cite{rordam-stable}.
Therefore our example shows that these results do not hold  even in quite basic nonseparable context,
as the $C^*$-algebra  $\A$ from Theorem \ref{main} satisfies $(*)$ while it  is nonstable. The latter is because  the multiplier algebra
 of any stable algebra with a projection 
contains a copy of $\B(\ell_2)$ (by 3.8. of \cite{multipliers}). However, $\A$ and therefore  $\tilde{A}$ (which is isomorphic to 
$\mathcal M(\A)$)
 are  scattered $C^*$-algebras as mentioned above,  and consequently all of their  subalgebras are
AF (\cite{kusuda-af}). Hence $\mathcal M(\mathcal A)$ does not contain a copy of $\B(\ell_2)$.
We need to add however, that a result of R\o rdam shows that
there are separable  extensions of $\K(\ell_2)$ which are not stable (\cite{rordam-nonstable}).

On a different note, it is worth noticing  that our $C^*$-algebra $\mathcal A$ is complemented  in the Banach
space $\mathcal M(\mathcal A)$ as it is co-one-dimensional closed subspace.
This fact does not hold for many nonseparable $C^*$-algebras (see 3.7 of \cite{taylor}).
It is also interesting to note that any separable subalgebra $\A_0$ of $\A$
is included in a separable subalgebra $\B\subseteq \A$ satisfying
$$0\rightarrow \mathcal K(\ell_2)\xrightarrow{\iota} \mathcal B 
\rightarrow\mathcal K(\ell_2)\rightarrow 0,$$
where $\iota[\mathcal K(\ell_2)]$ is essential. All such algebras $\B$ are
isomorphic to  $\widetilde{\mathcal K(\ell_2)}\otimes \mathcal K(\ell_2)$,
the noncommutative version of $C_0(\omega^2)$ (Proposition \ref{extensions-unique}), where $\omega^2$ is the ordinal
$\omega\times\omega$ with the order topology. Also note that
$\widetilde{\mathcal K(\ell_2)}\otimes 
\mathcal K(\ell_2(\cc))$ is a  stable $C^*$-algebra which satisfies the short exact sequence from $(*)$, 
and clearly is not isomorphic to our algebra which is not stable.
 These facts
have  well-known  analogues in the commutative context which is surveyed in \cite{hrusak}
devoted to applications of almost disjoint families in topology. One should add that there are many
noncommutative constructions
based on almost disjoint families (see the begining of Section 2.2 for the definition)  like in this paper
or in papers \cite{ilijas-oa}, \cite{vignati}, \cite{tristan}.

The structure of the paper is as follows. In Section 2 we recall and prove preliminary results
concerning liftings of ``systems of almost matrix units''  
$\mathcal T= \{T_{\eta, \xi}: \xi,\eta<\kappa\}\subseteq \mathcal B(\ell_2)$, which form systems
 of matrix units in the Calkin algebra.
The results are related to the liftings of families of almost orthogonal projections  (families of orthogonal projections  in the Calkin algebra),
which were analyzed in \cite{wofsey} and \cite{farah-wofsey}. Our $C^*$-algebra $\A$ from Theorem \ref{main} is generated by 
a specific ``maximal'' system of almost matrix units $\mathcal T$ and all operators in $\K(\ell_2)$.  
 A result of \cite{wofsey}
states that maximal almost disjoint families of subsets of $\N$ do not necessarily give rise to
 maximal families of almost orthogonal projections. This is enough to suggest that  Mr\'owka's original  almost disjoint family (\cite{mrowka}) can not be directly used for our purpose in the noncommutative setting.

In Section 3, for any system of almost matrix units $\mathcal T= \{T_{\eta, \xi}: \xi,\eta<\kappa\}\subseteq \mathcal B(\ell_2)$
 and an operator $R\in \B(\ell_2)$ which is a quasi-multiplier of $\A(\mathcal T)$, we assign a $\kappa \times \kappa$-matrix $\Lambda^\mathcal T (R)$.
The matrix $\Lambda^{\mathcal T}(R)$ carries a great load of information about $R$, 
and its analysis is crucial in the remaining
parts of the paper.  

In Section 4 we prove some results related to a system of almost matrix units labeled
by pairs of branches of the Cantor tree. In particular, it is essential later to use the Borel
structure of the standard topology on the Cantor tree in the form of the ``prefect
set property" of Borel subsets of the tree.

Section 5 is devoted to a simple method of modifying a system of almost matrix units
called  {\it pairing}. Finally in Section 6 we present the main construction which uses all the
previously developed theory.

The general scenario of the construction and the proof of
the properties of our algebra follows the main steps of \cite{mrowka}. However there are
two-fold complications. The usual problems related to passing from the commutative to
the noncommutative context, and the combinatorial difficulties related to the fact that
the objects corresponding to almost disjoint families, namely
the systems of almost matrix units,  are labeled by pairs and not single indices.
The  natural idea is to construct a system of almost matrix units $\mathcal S\subset \B(\ell_2)$ such that
the $C^*$-subalgebra $\A(\mathcal S)$  generated by $\K(\ell_2)$ and the elements of $\mathcal S$
 has no nontrivial (quasi-)multipliers, meaning that multipliers of the algebra $\A(\mathcal S)$ are the elements of 
$\mathcal A(\mathcal S)$ and the compact perturbations of the multiples of the identity. 
The   method  of eliminating (or ``killing" as it is usually called in set theory)
is the above-mentioned {pairing} from Section 6.

The notation and terminology should be standard and attempts to
follow texts like \cite{murphy}, \cite{invitation}, \cite{Blackadar}, \cite{farah-wofsey}.
For $T,S\in \mathcal B(\ell_2)$, we often write $T=^\K  S$ if $T-S\in 
\mathcal K(\ell_2)$. The map $\delta$ is always defined so that
$\delta_{\alpha, \beta}=1$ if $\alpha=\beta$ and $0$ otherwise.
 $[X]^{<\omega}$ denotes the family of all finite subsets of a set $X$ and
$[X]^2$ denotes the family of all two-elements sets of $X$.
For $C^*$-algebras $\A\subseteq \B(\ell_2)$ we identify the unitization
$\tilde{\A}$ with the subalgebra of $\B(\ell_2)$ generated by 
$\mathcal A$ and the identity operator $1_{\B(\ell_2)}$.

We would like to thank Hannes Thiel for bringing the paper \cite{taylor} to our attention
and to Ilijas Farah for pointing out some gaps in the previous versions of the paper and
 for valuable comments.

\section{preliminaries}

\subsection{Compact operators}
The following elementary lemma sums up the basic properties of the compact operators which will be used throughout this paper.

\begin{lemma}\label{compact-elementary}
Suppose that $\{e_n: n\in \N\}$ is an orthonormal basis for the Hilbert space $\ell_2$ 
and $S$ is
a bounded linear operator on $\ell_2$.
\begin{enumerate}
\item  If $\Sigma_{n\in \N} \|S(e_n)\|< \infty$,
then $S$ is compact,
\item If $S$
 is compact, $w_k\in span(e_n: n\in F_k)$ are  norm $1$ vectors, for pairwise disjoint finite
$F_k\subseteq \N$ and all $k\in \N$, then $(\|S(w_k)\|)_{k\in \N}\rightarrow 0$,
\item If $S$ is noncompact, then there is $\varepsilon>0$ such that
for every $k\in \N$ there is a finite subset $F_k\subseteq \N$ with $k<\min(F_k)$
and $w_k\in span(e_n: n\in F_k)$ of norm $1$ such that $\|S(w_k)\|>\varepsilon$.
\end{enumerate}
\end{lemma}
\begin{proof}
The above clauses easily follow from the fact that an operator $S$ is compact if and only if $\lim_{n\rightarrow \infty}
\|S(1 - R_n)\| = 0$, where $R_n$ is the projection on the span of $\{e_i : i\leq n\}$.
\end{proof}

Note that there are noncompact linear operators $S:\ell_2\rightarrow \ell_2$
satisfying $S(e_n)\rightarrow 0$. For example, consider the operator $S$ defined by
$S((x_n)_{n\in \N})(k)=\Sigma_{i\in I_k}{x_i/\sqrt{k}}$, where $(I_k)_{k\in\N}$
form pairwise disjoint consecutive intervals in $\N$ of size $k$. Considering
$w_k={1\over\sqrt k}\chi_{I_k}$, where $\chi_{I_k}$ is the characteristic function on $I_k$,
 one  can easily verify that (2) fails.

\subsection{Families of almost orthogonal projections}

A family $\{A_\xi: \xi < \kappa\}$ of subsets of $\N$ is called an almost disjoint family if $A_\xi\cap A_\eta$ is finite 
for distinct $\xi,\eta< \kappa$. Suppose $\wp(\N)$ denotes the Boolean algebra of all subsets of $\mathbb N$ and $Fin$ is the ideal of all finite subsets of $\mathbb N$.  Almost disjoint families correspond to sets of pairwise incomparable elements (antichains) of $\wp(\N)/Fin$. 
An almost disjoint family is called maximal if it is maximal with respect to the inclusion.
For a fixed orthonormal basis for the Hilbert space $\ell_2$,  $\wp(\N)/Fin$ naturally embeds in the poset of 
projection of the Calkin algebra. In other words, any almost disjoint family $\{A_\xi: \xi < \kappa\}$   would naturally give rise to 
a family of diagonalized projections $\{P_\xi: \xi < \kappa\}$ on $\ell_2$ such that $P_\xi P_\eta$ is a compact 
(finite dimensional) projection, 
for any distinct $\xi,\eta <\kappa$. The following is a natural generalization of such families.

\begin{definition}[\cite{wofsey}] For a Hilbert space $\mathcal H$, a  family $\mathcal P$ of noncompact projections
of $\mathcal B(\mathcal H)$ is called almost orthogonal if the product of any two distinct elements of it
is compact. Such a family $\mathcal P$ is called maximal if for every 
noncompact projection $Q\in \mathcal B(\mathcal H)$
the operator $PQ$ is noncompact, for some $P\in \mathcal P$. 
\end{definition}

Having fixed an orthonormal basis $(e_n: n\in \N)$
for $\ell_2(\N)$ and given a family $\mathcal F\subseteq \wp(\N)$ one can 
consider the orthogonal projections $P_A$ for $A\in \mathcal F$ 
onto the closed span of $\{e_n: n\in A\}$.
As it was observed in \cite{wofsey},  almost orthogonal families of projections corresponding 
in the above sense
to maximal 
almost disjoint families do not have to be maximal. 

 Recall that a ``masa" of $\B(\ell_2)$ is a maximal abelian subalgebra
of $\B(\ell_2)$. A masa is called atomic if it is isomorphic to $\ell_\infty$, the algebra of  all operators that are diagonalized by a fixed basis for $\ell_2$.
The following is Lemma 5.34 of  \cite{farah-wofsey}.
\begin{lemma}\label{5.34} Let
$\pi: \mathcal B(\ell_2)\rightarrow \mathcal B(\ell_2)/\mathcal K(\ell_2)$ 
be the quotient map. Given any sequence  $\{P_n: n\in \N\}$ of
projections in $\B(\ell_2)$ such that $\pi(P_i)$ and  $\pi(P_j)$ commute for all $i, j\in \N$, there is
an atomic masa $\mathcal A$ in $\B(\ell_2)$ such that $\pi[\A]$ contains each $\pi(P_i)$
for $i\in \N$.
\end{lemma}

\begin{lemma}\label{countable-diagonalization}
 Suppose that $\{P_n: n\in \N\}$ is  an almost orthogonal family of projections
of $\mathcal B(\ell_2)$.  Then there are pairwise orthogonal 
projections $\{R_n: n\in \N\}$ in $\mathcal B(\ell_2)$ such that $P_n=^\K  R_n$,
 for every $n\in \N$.
\end{lemma}
\begin{proof}
By Lemma \ref{5.34} 
there is an atomic masa $\mathcal A$ in $\mathcal B(\ell_2)$
such that $\pi[\mathcal A]\supseteq \{\pi(P_n): n\in \N\}$. Since $\mathcal A$ is isomorphic to 
 $\ell_\infty\cong C(\beta\N)$, 
the ideal of compact operators in $\mathcal A$ is isomorphic to
$c_0\cong  C_0(\beta\N, \N^*)=\{f\in C(\beta\N): f|\N^*=0\}$, where $\N^*=\beta\N\setminus \N$. 
 Therefore $\mathcal A/(\mathcal K(\ell_2)\cap \A)\cong C(\beta\N)/ C_0(\beta\N, \N^*)\cong C(\N^*)$. 
As $\pi(P_n)$ are orthogonal projections, 
they correspond to the characteristic functions of pairwise disjoint clopen
subsets of $\N^*$. Such sets are given by pairwise disjoint elements of
$\wp(\N)/Fin$. For any such family in $\wp(\N)/Fin$ we can choose pairwise disjoint representatives
in $\wp(\N)$, which define disjoint clopen subsets of $\beta\N$ and therefore
pairwise orthogonal projections $R_n$ in $\mathcal A$ such that $\pi(R_n)=\pi(P_n)$.
\end{proof}
In the following $R|X$ denotes the restriction of the operator $R$ to the closed subspace $X$ of $\ell_2$.

\begin{lemma}\label{bounded-below} Suppose that $R\in \mathcal B(\ell_2)$ is noncompact
and self-adjoint.
Then there is a closed infinite dimensional subspace $ X\subseteq \ell_2$ such that
$R|X$ is invertible in $\B(X)$ and $R$ commutes with the orthogonal projection
$P_X$ onto $X$.
\end{lemma}
\begin{proof} By the spectral theorem there are a measure space $(\mathcal M, \mu)$, an 
isomorphism of Hilbert spaces 
$U: \ell_2\rightarrow L_2(\mathcal M, \mu)$, and a measurable function $f$ such that
$URU^*$  is equal to the operator $M_f$ on
 $ L_2(\mathcal M, \mu)$ acting by  multiplication by $f$.
 Since $R$ is noncompact, we have $f\neq 0$. 
Putting $A_n=\{x\in \mathcal M: |f(x)|>1/n\}$, we have 
$M_f= \lim_{n\rightarrow \infty} M_{f\chi_{A_n}}$,
where the convergence is in the operator norm as $f-f\chi_{A_n}$
is a function bounded by $1/n$. Let 
$P_n=U^* M_{\chi_n} U$. 
Each $P_n$ is a projection which commutes
with $R$ and  $R$  is bounded away from zero on the range of $P_n$ for
each $n\in \N$. Moreover
 $(RP_n)_{n\in \N}$  converges (in the norm) to $R$.
 As $R$ is noncompact, for some $n$ the projection $P_n$ is infinite dimensional. Hence we
obtain the lemma by letting $X = P_n [\ell_2]$. 
\end{proof}

\begin{lemma}\label{maximal-noncompact} Suppose that $\mathcal P$ is a maximal
almost orthogonal family of projections 
of $\mathcal B(\ell_2)$ and $S\in \mathcal B(\ell_2)$ 
is a self-adjoint and noncompact operator. Then there are $P_1, P_2\in \mathcal P$ such that
$P_2 SP_1$ is noncompact.
\end{lemma}
\begin{proof} Let $X\subseteq \ell_2$ be an infinite dimensional
subspace such that $S|X$ is invertible and  $P_X$ commutes with $S$, which exists by Lemma
\ref{bounded-below}. 
By the maximality of $\mathcal P$ we find $P_1\in \mathcal P$
such that $P_XP_1$ is noncompact. Therefore $P_XP_1P_X$ is a self-adjoint and noncompact operator.
 Using  Lemma
\ref{bounded-below} again for $P_XP_1P_X$,  there is an infinite dimensional subspace $Y\subseteq X$
such that $(P_XP_1P_X)|Y$ is invertible. So $P_XP_1P_X$ acts on $Y$
as an isomorphism of Banach spaces, transforming $Y$ into its
image $(P_XP_1 P_X)[Y]$ which is an infinite dimensional subspace of $X$.
Since $S$ acts as an isomorphism of Banach spaces on $X$, it follows that
$SP_XP_1P_X$ is noncompact. Also since $S$ commutes with $P_X$, the operator 
 $P_XSP_1P_X$ is noncompact, and therefore $SP_1$ is noncompact. 
Working with $SP_1$ instead of $S$, similarly  we find $P_2\in \mathcal P$ such that
$P_2SP_1$ is noncompact.
\end{proof}

\subsection{Systems of almost matrix units}

Let $\kappa$ be a cardinal and  $\A$ be a $C^*$-algebra.
A family 
 $\{a_{\beta,\alpha}: \alpha,\beta<\kappa\}$ of nonzero elements of $\mathcal A$
satisfying the following matrix units relations:
\begin{itemize}
\item $a_{\beta,\alpha}^*=a_{\alpha,\beta}$ for all $\alpha,\beta<\kappa$,
\item
$a_{\beta,\alpha} a_{\gamma,\eta}=\delta_{\alpha,\gamma}a_{\beta,\eta}$
for all $\alpha, \beta, \gamma, \eta<\kappa$,
\end{itemize}
is called a system of matrix units in $\mathcal A$.

\begin{proposition}\label{matrixunitsiso} Let  $\A$ be 
the $C^*$-algebra
 generated by a system of its matrix units $\{a_{\eta, \xi}: \xi, \eta<\kappa\}$.
Then $\A$ is $^*$-isomorphic to the algebra $\K(\ell_2(\kappa))$
of all compact operators on $\ell_2(\kappa)$.
\end{proposition}
\begin{proof}  Let $\{e_\xi :\xi<\kappa\}$ be an orthonormal basis for $\ell_2(\kappa)$ and the operators
$\{T_{\eta, \xi} : \xi,\eta< \kappa\}$  are the system of matrix units in $\B(\ell_2(\kappa))$ defined by 
$T_{\eta, \xi}(e_\xi)= e_\eta$ and $T_{\eta, \xi}(e_\xi')= 0$ for $\xi'\neq \xi$.  For every finite subset $F$
of $\kappa$, let $\B_F$ be the $C^*$-subalgebra generated by $\{T_{\eta, \xi} : \xi,\eta\in F\}$, which is clearly
isomorphic to $M_{|F|}$, the algebra of all $|F|\times |F|$ matrices. 
 Let $\B$
be the inductive limit of the algebras $\B_F$, along the set $\mathcal F_\kappa$ of finite subsets of $\kappa$ and the *-homomorphisms 
$\phi_{G,F}: \B_F \rightarrow \B_G$ for $F\subseteq G$ and $F,G\in \mathcal F_\kappa$, defined by 
$\phi_{G,F}(T_{\eta,\xi})= T_{\eta,\xi}$, for $\xi,\eta\in F$.  Clearly $\B= \K(\ell_2(\kappa))$. The map which 
sends $a_{\eta,\xi}$ to $T_{\eta,\xi}$ extends to a $*$-isomorphism from  $\A$ onto $\B$.
\end{proof}

\begin{definition}\label{almostunits} Suppose that
$\mathcal T=\{T_{\eta, \xi}: \xi, \eta<\kappa\}\subseteq \mathcal B(\ell_2)$
is  a family of noncompact operators.
We say that $\mathcal T$ is a system   of
almost matrix units  if and only if 
for every $\alpha, \beta, \xi, \eta<\kappa$,
\begin{enumerate}
\item $T_{\eta, \xi}^* =^\K   T_{\xi, \eta}$,
\item $T_{\beta, \alpha}\ T_{\eta, \xi} =^\K  \delta_{\alpha, \eta}T_{\beta, \xi}$.
\end{enumerate}
\end{definition}

\begin{definition}\label{maximal-system} 
Suppose that $\mathcal T=\{T_{\eta, \xi}: \xi, \eta<\kappa\}\subseteq \mathcal B(\ell_2)$
is a system of almost matrix units and $\{P_\xi: \xi<\kappa\}$ is a collection of
almost orthogonal projections in $\mathcal B(\ell_2)$.
We say that   $\mathcal T$ is based on $\mathcal P$ 
if $T_{\xi, \xi}=^\K P_\xi$ for all $\xi<\kappa$. We say that $\mathcal T$ is a 
maximal system of almost matrix units
  if it is based on a maximal family of almost orthogonal projections.
\end{definition}

\begin{lemma}\label{based-on-projections} Every system of almost matrix units is based on a family of almost
orthogonal projections.
\end{lemma}
\begin{proof} By the almost matrix units relations \ref{almostunits} we have that
$T_{\xi, \xi}^*=^\K T_{\xi, \xi}$ and  
$T_{\xi, \xi}T_{\xi, \xi}=^\K T_{\xi, \xi}$, so
$[T_{\xi, \xi}]_{\K(\ell_2)}$ is a projection in the Calkin algebra. Therefore we can find
a projection $P_\xi\in \mathcal B(\ell_2)$ such that $P_\xi=^\K T_{\xi, \xi}$
(see Lemma 5.3. of \cite{farah-wofsey}).
\end{proof}

In the rest of this section  we use some elementary facts about partial isometries i.e.,
elements of $\B(\ell_2)$ which are isometries on a subspace of $\ell_2$ and zero
on its orthogonal complement.  For an element $U\in \B(\ell_2)$ being a partial
isometry is equivalent to each of the conditions (i) $U=UU^*U$, (ii) $U^*=U^*UU^*$, (iii)
$U^*U$ is a projection, (iv) $UU^*$ is a projection (2.3.3. \cite{murphy}). Moreover recall that by the 
polar decomposition,  any $T\in \B(\ell_2)$ can be written as $T=U|T|$, where
$U$ is a partial isometry whose kernel is equal to the kernel of $T$ (2.3.4. \cite{murphy}).

\begin{lemma}\label{system-existence}
Suppose that $\mathcal P=\{P_\xi: \xi<\kappa\}\subseteq \mathcal B(\ell_2)$ is a family
of almost orthogonal projections. Then there is a system of
almost matrix units $\mathcal T$ based on $\mathcal P$.
\end{lemma}
\begin{proof} Since  $P_\xi$ for $\xi<\kappa$ are  infinite dimensional projections, there are partial isometries  
 $T_{\xi, 0}\in \mathcal B(\ell_2)$  such that 
$T_{\xi, 0}^* T_{\xi, 0} = P_0$ and $T_{\xi, 0} T_{\xi, 0}^* = P_\xi$, for each $\xi<\kappa$.
Let  $T_{0,\xi} = T_{\xi,0}^*$. 
We have $T_{\xi, 0}=P_\xi T_{\xi, 0}$ and $T_{0,\xi} =T_{0,\xi} P_\xi$.
For  $\xi, \eta<\kappa$,
 define 
$$T_{\xi, \eta}=T_{\xi, 0}T_{0,\eta}.$$ 
 It is clear that $\{T_{\xi, \eta}: \xi, \eta<\kappa\}$ satisfies the condition $(1)$ of Definition 
\ref{almostunits}. For $(2)$ note that if $\alpha,\eta < \kappa$ then 
$$ T_{0,\alpha}T_{\eta, 0}= T_{0,\alpha}P_\alpha P_\eta T_{\eta, 0}, $$
which is compact if $\xi\not=\eta$, by the almost orthogonality of $P_\xi$s and
 $T_{\beta, \alpha}T_{\alpha, \xi}=T_{\beta, \xi}$.
\end{proof}

\begin{lemma}\label{extension-to-maximal} Every system of almost matrix units can be extended
to a maximal one.
\end{lemma}
\begin{proof}Suppose that $\mathcal T=\{T_{\eta, \xi}: \xi, \eta<\kappa\}$ is a system 
of almost matrix units. Let $\mathcal P=\{P_\xi: \xi \in \kappa\}$  be a family of projections such that
$T_{\xi, \xi}=^\K P_\xi$ as in Lemma \ref{based-on-projections}.  Extend 
$\mathcal P$ to a maximal
family of almost orthogonal projections $\mathcal P'=\{P_\xi: \xi \in \kappa\}\cup\{P_\xi: \xi\in X\}$  for some set $X$ disjoint from $\kappa$. 
Use Lemma \ref{system-existence} to construct a system of almost
matrix units $\{T_{\eta, \xi}: \xi, \eta\in X\cup\{0\}\}$
based on $\{P_\xi: \xi\in X\cup\{0\}\}$. 
For $\xi\in \kappa$ and $\eta\in X$ define
$$T_{\eta, \xi}= T_{\eta, 0} T_{0,\xi}, \ \ T_{\xi, \eta}=T_{\xi, 0}T_{0, \eta}.$$

It is straightforward to check that $\{T_{\eta, \xi}: \xi, \eta\in \kappa\cup X\}$
forms a system of almost matrix units based on $\mathcal P'$.

\end{proof}

The  next lemma is a version of Lemma III 6.2 from \cite{davidson}.
\begin{lemma}\label{lifting-compact} 
Suppose that $\{T_{j, i}: i, j\in \N\}$ is a system of almost matrix units in $\B(\ell_2)$.
Then there is  a system $\{E_{j, i}: i, j\in \N\}$
 of matrix units in $\B(\ell_2)$ such that 
$E_{j, i}=^\K T_{j, i}$ for every $i, j\in \N$.
\end{lemma}
\begin{proof} Let $\A$ be a $C^*$-algebra generated in $\B(\ell_2)$
by $\{T_{j, i}: i, j\in \N\}$ and the compact operators. Then since $\mathcal B = \A/\K(\ell_2)$ is  generated by $\{[T_{j,i}]_{\K(\ell_2)}\in \mathcal A/\K(\ell_2): i,j\in \N \}$, it is isomorphic to $\K(\ell_2)$
(see Lemma \ref{short-exact-sequence}).
Since $K_0(\B)\cong \mathbb Z$ is a free abelian group, $\A$ is a trivial extension (see Exercise  16.4.7 of \cite{blackadar-k}), i.e., the short exact sequence 
$$0\rightarrow \mathcal K(\ell_2)\xrightarrow{\iota} \mathcal A 
\rightarrow\B\rightarrow 0$$
splits, which means $\{[T_{j,i}]_{\K(\ell_2)}\in \B: i,j\in \N \}$ lift.
\end{proof}

\subsection{$\Psi$-type $C^*$-algebras}

If $\mathcal T=\{T_{\xi, \eta}: \xi, \eta<\kappa\}$ is a system of almost matrix units,
we use $\mathcal A(\mathcal T)$ to denote the $C^*$-subalgebra of
$\mathcal B(\ell_2)$ generated by $\{T_{\xi, \eta}: \xi, \eta<\kappa\}$ and
the compact operators in 
$\mathcal B(\ell_2)$. 

\begin{lemma}\label{short-exact-sequence}
 Suppose that $\mathcal T=\{T_{\xi, \eta}: \xi, \eta<\kappa\}$ is a system of almost matrix units
in $\B(\ell_2)$.
The $C^*$-algebra $\mathcal A(\mathcal T)$
 satisfies the short
exact sequence
$$0\rightarrow \mathcal K(\ell_2)\xrightarrow{\iota} \mathcal A(\mathcal T)\xrightarrow{\pi}
 \mathcal K(\ell_2(\kappa))\rightarrow 0,$$
where ${\iota}[\mathcal K(\ell_2)]$ is an essential ideal of $\mathcal A(\mathcal T)$.
If $\kappa$ is uncountable, then the extension is not split, i.e., 
there is no $\sigma:\K(\ell_2(\kappa))\rightarrow \mathcal A(\mathcal T)$
such that $\pi\circ \sigma$ is the identity on $\K(\ell_2(\kappa))$.
\end{lemma}
\begin{proof} The map $\iota$ is the inclusion. 
 Since $\mathcal K(\ell_2)\subseteq \mathcal A(\mathcal T)\subseteq \B(\ell_2)$ and
$\mathcal K(\ell_2)$ is an essential ideal in $\B(\ell_2)$, we conclude that
$\iota[\K(\ell_2)]$ is an essential ideal of $A(\mathcal T)$.

The operators
 $\{[T_{\xi, \eta}]_{\K(\ell_2)}\in \mathcal A/\K(\ell_2): \xi, \eta <\kappa \}$
generate $\mathcal A(\mathcal T)/ \mathcal K(\ell_2)$ 
(by the definition of $\mathcal A(\mathcal T)$) and satisfy the matrix unit 
relations in $\B(\ell_2)/\mathcal K(\ell_2)$, that is 
\begin{itemize}
\item $[T_{\xi,\eta}]_{\K(\ell_2)}^* = [T_{\eta,\xi}]_{\K(\ell_2)}$, 
\item $[T_{\beta, \alpha}]_{\K(\ell_2)}
 [T_{\eta, \xi}]_{\K(\ell_2)}=\delta_{\alpha, \eta}[T_{\beta, \xi}]_{\K(\ell_2)}$.
\end{itemize}
Thus $\mathcal A(\mathcal T)/\mathcal K(\ell_2) \cong \mathcal K(\ell_2(\kappa))$ by Proposition
\ref{matrixunitsiso}.
If $\kappa$ is uncountable, then we observe that $\K(\ell_2(\kappa))$ can not be embedded
into $\B(\ell_2)$ and so it can not be embedded into $\mathcal A(\mathcal T)$. This
follows from the fact that $\B(\ell_2)$ does not contain any uncountable family of pairwise orthogonal projections, while  $\K(\ell_2(\kappa))$ clearly does, if $\kappa$ is uncountable.
\end{proof}

We say a $C^*$-algebra is \emph{$\Psi$-type} if it is of the form $\A (\mathcal T)$ for a system of almost matrix units $\mathcal T$.
These $C^*$-algebras are the natural noncommutative analogues of the $\Psi$-spaces
in topology, which are induced by almost disjoint families (see Definition 2.6 of \cite{hrusak}). 
In topology $\Psi$-spaces are 
classical examples of separable locally compact Hausdorff scattered (every nonempty subset has a relative isolated point)
 spaces with the Cantor-Bendixson height two. Granting the role of isolated points to minimal projections in  $C^*$-algebras,
one can define scattered  $C^*$-algebras. 
A projection $p$ in a $C^*$-algebra $\A$ is called minimal if $p\A p = \mathbb C p$ and 
a  $C^*$-algebra is scattered if every nonzero subalgebra has a minimal projection  
(see \cite{scattered} for more on scattered $C^*$-algebras). Just like the
scattered spaces, these algebras can be analyzed using the ``Cantor-Bendixson sequences". 
For a $C^*$-algebra $\A$ let $\I^{At}(\mathcal A)$ denote the subalgebra of $\mathcal A$ generated by 
the minimal projections of $\mathcal A$. The subalgebra $\I^{At}(\mathcal A)$ turns out to be an ideal isomorphic 
to a subalgebra of all compact operators in any faithful representation of $\A$ and 
in fact is the largest ideal with this property (Proposition 3.15 and Proposition 3.16 of \cite{scattered}). 
A $C^*$-algebra $\A$ is  scattered if and only if there is an ordinal $ht(\mathcal A)$
and a  increasing sequence of closed ideals $(\I_\alpha)_{\alpha\leq ht(\mathcal A)}$
such that  $\I_0=\{0\}$, $\I_{ht(\mathcal A)}= \mathcal A$, if $\alpha$ is a limit ordinal 
$\I_{\alpha}= \overline{\bigcup_{\beta<\alpha} \I_{\beta}}$, and
$$\I_{\alpha+1}/ \I_{\alpha} = \I^{At}(\mathcal A /\I_\alpha),$$
for every $\alpha < ht(\mathcal A)$ (Theorem 1.4 of \cite{scattered}).
The sequence $(\I_\alpha)_{\alpha\leq ht(\mathcal A)}$ is called the Cantor-Bendixson sequence
for $\A$ and the ordinal $ht(\A)$ is called the Cantor-Bendixson height or simply the height of $\A$.
\begin{proposition} Suppose that $\mathcal T=\{T_{\xi, \eta}: \xi, \eta<\kappa\}$ is a system of almost matrix units.
The $C^*$-algebra $\mathcal A(\mathcal T)$
 is a scattered $C^*$-algebra of height $2$.
Therefore $\A(\mathcal T)$ is GCR (type I) and AF, in the sense that every
finite set of elements can be approximated from a finite dimensional subalgebra.
\end{proposition}
\begin{proof} 
 Since $\mathcal K(\ell_2)\subseteq \mathcal A(\mathcal T)\subseteq \B(\ell_2)$, by Proposition 3.21
of \cite{scattered} 
 we have
 $\I_1= \I^{At}(\mathcal A(\mathcal T))= \mathcal K(\ell_2)$
 and also $ \I_2/\I_1 = \I^{At}(\A/ \I^{At}(\mathcal A(\mathcal T)) 
\cong
 \mathcal K(\ell_2(\kappa))$
 by Lemma \ref{short-exact-sequence},
and therefore $\I_2=\A(\mathcal T)$.

The composition series $(0, \K(\ell_2), \A(\mathcal T))$ witnesses the fact that 
 $\mathcal A(\mathcal T)$ is GCR (see IV.1.3 of \cite{Blackadar}). 
 Also every scattered $C^*$-algebra is AF (see \cite{lin}, cf. \cite{scattered}).
\end{proof}

Let us conclude this section by observing the contrast between the separable and nonseparable
 case for the extensions of an algebra of compact operators by compact operators.

\begin{proposition}\label{extensions-unique}
Suppose that $\mathcal B$ is a  $C^*$-algebra satisfying the short exact sequence
$$
0\rightarrow \K(\ell_2)\xrightarrow{i} \B \xrightarrow{q} \K(\ell_2) \rightarrow 0,
$$
where $i[\K(\ell_2)]$ is an essential ideal of $\B$. Then $\B$ is *-isomorphic to 
 $\widetilde{\mathcal K(\ell_2)}$
 $\otimes \mathcal K(\ell_2)$.
\end{proposition}

\begin{proof}
It is enough to show that $\B$ is unique up to *-isomorphism. Since $K_0(\K(\ell_2))$ is a free abelian group the sequence
above splits (Exercise  16.4.7 of \cite{blackadar-k}). All the nonunital split essential extensions of 
$\K(\ell_2)$ by $\K(\ell_2)$ are equivalent and therefore isomorphic (see II.8.4.30 of \cite{Blackadar}).

\end{proof}

\section{Multipliers of systems of  almost matrix units}

Let $\A$ be a  nondegenerate subalgebras  of $\B(\ell_2)$. A multiplier of (or a multiplier for) $\A$ is an operator $T$ in $\B(\ell_2)$ such that 
$T\mathcal A\subseteq \mathcal A$ and  $ \mathcal A T \subseteq \A$. An operator $T$ in $\B(\ell_2)$ is called a quasi-multiplier of $\A$ if $\A T \A \subseteq \A$. We denote the set of multipliers of $\A$ by $\mathcal M(\A)$ and the set of all quasi-multipliers of $\A$ by $\mathcal Q\mathcal M(\A)$. It is well-known that $\mathcal Q\mathcal M(\A)$ is a norm closed *-invariant subspace of $\A''$ and $\mathcal M(\A)$ is a $C^*$-subalgebra of $\A''$ and of course $\A \subseteq \mathcal M(\A) \subseteq \mathcal Q\mathcal M(\A)$ (see 3.12 of \cite{Ped}).

\begin{lemma}\label{multipliers-for-T} Suppose that $\mathcal T=\{T_{\eta, \xi}: \xi, \eta<\kappa\}
\subseteq \mathcal B(\ell_2)$
is a system of almost matrix units. Then the following are equivalent:
\begin{enumerate}
\item $R\in \mathcal Q\mathcal M(\A(\mathcal T))$,
\item for every $\xi, \eta<\kappa$ there is $\lambda_{\xi, \eta}^\mathcal T(R)\in \C$ such that
$$T_{\eta, \eta}R  T_{\xi, \xi}=^\K  \lambda_{\eta, \xi}^\mathcal T(R) T_{\eta, \xi}.$$
\end{enumerate}
\end{lemma}
\begin{proof}
Suppose that $R$ is a quasi-multiplier of $\A(\mathcal T)$ and $\xi, \eta<\kappa$ are given. Then $S =T_{\eta, \eta}R  T_{\xi, \xi}$ is an operator in $\A$  satisfying $S=^{\mathcal K}T_{\eta, \eta}ST_{\xi, \xi}$. The only operators in $\A$   with this property are compact perturbations of constant multiples of $T_{\eta, \xi}$. The other implication follows immediately from the definition of $\A(\mathcal T)$. 
\end{proof}

 If $R$ is a quasi-multiplier of $\A(\mathcal T)$, let $\Lambda^\mathcal T (R)$ denote the $\kappa\times \kappa$ matrix 
$(\lambda_{\eta, \xi}^{\mathcal T}(R))_{\xi,\eta<\kappa}$ over $\mathbb C$. 
If $\mathcal T$ is clear from the context, 
we often drop the superscript $\mathcal T$, and write 
$\Lambda(R) = (\lambda_{\eta, \xi}(R))_{\xi,\eta<\kappa}$. 

In the following, we use $I_\kappa$ to denote 
the $\kappa \times \kappa$ matrix which has constant $1$
 on the diagonal and zero everywhere else,  where $\kappa$ is a cardinal.
 We will also use $1_{B (\ell_2)}$ to denote the unit element
 of $\mathcal B (\ell_2)$. When considering $\kappa\times\kappa$
matrices we can treat some of them as operators in $\mathcal B (\ell_2(\kappa))$. 
Namely, for a fixed (the canonical) orthonormal basis $\{e_\xi: \xi<\kappa\}$ for $\ell_2(\kappa)$, we identify
 operators $T_M \in \B(\ell_2(\kappa))$ defined by  $T_M(e_\xi)(\eta)=m_{\eta, \xi}$ with the 
$\kappa\times\kappa$ matrix  $M=(m_{\xi, \eta})$. So, for example,
$T_{I_\kappa}$ is the unit of $\mathcal B (\ell_2(\kappa))$ and
$T_M$ is compact if $M$ is a matrix which has only finitely many nonzero entries.
In particular, we will say that a matrix $M$ is a matrix of a compact operator 
if $T_M$ is compact. The operations of addition, multiplication by scalar and the
transposition of $\kappa\times\kappa$ matrices should be clear.

\begin{lemma}\label{rows} Assume $\mathcal T$ is
   a  system of almost 
matrix units of size $\kappa$ and $R\in \B(\ell_2)$ is 
a quasi-multiplier of $\A(\mathcal T)$. Then $\Lambda^{\mathcal T}(R)$ is a matrix 
of a bounded linear operator on $\ell_2(\kappa)$ of norm not
bigger than $\|R\|$.  In particular, all 
 rows and columns of the matrix $\Lambda^{\mathcal T}(R)$ are in $\ell_2(\kappa)$.
\end{lemma}
\begin{proof}
It is enough to prove that for any finite $F\subseteq \kappa$ and for any
$(c_\xi)_{\xi\in F}\subseteq \C$ such that $\Sigma_{\xi\in F}|c_\xi|\leq 1$
we have  
$$\sqrt{\Sigma_{\eta\in F}|\Sigma_{\xi\in F}\lambda_{\eta, \xi}c_\xi|^2}\leq\|R\|.\leqno (*)$$
Using Lemma \ref{lifting-compact} we have a system of matrix units $(E_{\eta, \xi})_{\eta, \xi \in F}$ 
in $\B(\ell_2)$ such that $T_{\eta, \xi}=^\K E_{\eta, \xi}$ for every $\xi, \eta\in F$.
It follows that 
$$E_{\eta, \eta}RE_{\xi, \xi}=\lambda_{\eta, \xi}(R)E_{\eta, \xi}+S_{\eta, \xi},$$
where $S_{\eta, \xi}$ is a compact operator, for each $\xi,\eta\in F$. 
For a given $\epsilon>0$ we will find a norm one vector $w\in \ell_2$ such that 
$\|R(w)\|^2\geq \Sigma_{\eta\in F}|\Sigma_{\xi\in F}\lambda_{\eta, \xi}c_\xi|^2-\varepsilon$,
which will prove $ (*)$.

By considering an infinite orthonormal basis  in the ranges of each $E_{\xi, \xi}$ 
for $\xi\in F$ and using Lemma \ref{compact-elementary} (2)  we can find norm 1 vectors
$w_\xi$ in the ranges of $E_{\xi,\xi}$, respectively, such that
$$\Sigma_{\eta\in F}\Sigma_{\xi\in F}|c_\xi|^2 \|S_{\eta, \xi}(w_\xi)\|^2<\varepsilon,$$
and  $w_{\eta} = E_{\eta,\xi}(w_\xi)$ for  $\xi,\eta\in F$. The last statement follows from the fact that
$E_{\eta, \xi}$s are partial isometries, so all the orthonormal bases  may be
considered to be the images of a fixed orthonormal basis in $E_{\eta, \eta}$. 

So by the pairwise orthogonality of $E_{\xi, \xi}$s for $\xi\in F$ and by the 
Pythagorean theorem we have 
\begin{align*}
\|R(\Sigma_{\xi\in F}c_\xi w_\xi)\|^2
\geq \Sigma_{\eta\in F}|\Sigma_{\xi\in F}\lambda_{\eta, \xi}c_\xi|^2   -  \varepsilon,
\end{align*}
which completes the proof.
\end{proof}

In particular  by Lemma \ref{rows} all columns and rows  can have at most countably many nonzero entries. 
Therefore if  $\kappa$ is an uncountable cardinal and 
  $R$ is a  quasi-multiplier of $\A(\mathcal T)$,
then for every $\xi<\kappa$ there is $\eta<\kappa$ such that
$\lambda_{\eta, \xi}^{\mathcal T}(R)=0$.

\begin{lemma}\label{multipliers-operations} Assume $\mathcal T$ is
   a  system of almost 
matrix units of size $\kappa$. The map $\Lambda^{\mathcal T}$
 from $\mathcal Q \mathcal M(\A(\mathcal T))$ into $\B(\ell_2(\kappa))$
is a norm one  linear operator such that $\Lambda^{\mathcal T}(R^*)=\Lambda^{\mathcal T}(R)^*$ for every quasi-multiplier
$R$ of $\A(\mathcal T)$.
\end{lemma}
\begin{proof}
The linearity of $\Lambda^{\mathcal T}$ is  immediate.
The fact that $\|\Lambda^{\mathcal T}\|\leq 1$ follows from Lemma \ref{rows}.
For the last part, note that $T_{\eta, \eta}R^*T_{\xi, \xi}=
(T_{\xi, \xi}RT_{\eta, \eta})^*=(\lambda_{\xi, \eta}^{\mathcal T}(R)T_{\xi, \eta})^*=
\overline{\lambda_{\xi, \eta}^{\mathcal T}(R)}T_{\eta, \xi}$, and therefore  
$\lambda_{\eta, \xi}^{\mathcal T}(R^*)=\overline{\lambda_{\xi, \eta}^{\mathcal T}(R)}$.
\end{proof}

\begin{lemma}\label{compact-in-at} Suppose that $\mathcal T=\{T_{\eta, \xi}: \xi, \eta<\kappa\}
\subseteq \mathcal B(\ell_2)$
is a system of almost matrix units.  For every  $\kappa\times\kappa$ matrix 
$(\lambda_{\eta, \xi})_{\xi, \eta<\kappa}$ of a compact operator on $\ell_2(\kappa)$
there is  $R\in \A(\mathcal T)$ such that
$\lambda_{\eta, \xi}(R)=\lambda_{\eta, \xi}$ for every
$\xi, \eta<\kappa$ (see Definition \ref{multipliers-for-T}).

\end{lemma}
\begin{proof} Since $(\lambda_{\eta, \xi})_{\xi, \eta<\kappa}$
is  a matrix of a compact operator on $\B(\ell_2(\kappa))$ (denote this operator  by $S$),
there is a countable $A\subseteq \kappa$ such that
$\lambda_{\xi, \eta}=0$ if $(\xi, \eta)\not\in A\times A$. 
This follows from the fact  that the image of the unit ball under a compact  operator
is compact and metrizable, and hence separable which implies that the matrix of
the operator must have at most countably many nonzero rows. Now since each row of the matrix 
$(\lambda_{\eta, \xi})_{\xi, \eta<\kappa}$ belongs to $\ell_2(\kappa)$, there can be at most countably
many nonzero entries.

Apply Lemma \ref{lifting-compact} to obtain a system of matrix units $(E_{\eta, \xi})_{\xi, \eta\in A}$
in $\B(\ell_2)$ such that $E_{\eta, \xi}=^\K T_{\eta, \xi}$ for $\xi, \eta\in A$. 
Let $(F_n)_{n\in  \N}$ be
an increasing  sequence of finite sets such that $\bigcup_{n\in \N}F_n=A$.
Let $\B_n$ be the subalgebra of 
$(\Sigma_{\xi\in F_n}E_{\xi, \xi})\B(\ell_2)(\Sigma_{\xi\in F_n}E_{\xi, \xi})$ of all operators
of the form 
$$\Sigma_{\xi, \eta\in F_n}\alpha_{\eta,\xi}E_{\eta, \xi},\leqno (*)$$
where $\alpha_{\eta, \xi}\in \C$ for every $\xi, \eta\in F_n$.
From the matrix unit relations and Proposition \ref{matrixunitsiso} it 
follows  that
$\B_n$ is $*$-isomorphic
to  the algebra of $|F_n|\times |F_n|$
matrices. Therefore the norm of the operator as in $(*)$ is equal to
the matrix norm of the matrix $(\alpha_{\eta,\xi})_{\xi, \eta}$. The 
norm of this operator in $\B(\ell_2))$ is the same.

Let $S_n\in \B(\ell_2(\kappa))$ be given by $S_n=P_{F_n}SP_{F_n}$, where
$P_X$ is the orthogonal projection from $\ell_2(\kappa)$ onto $\ell_2(X)$
for $X\subseteq \kappa$. Then
since $(\lambda_{\eta,\xi})_{\xi, \eta}$ is a matrix of a compact operator $S$, the
operators $S_n$
  converge in the norm to $S$. 
Consider the operators 
$$R_n=\Sigma_{\xi, \eta\in F_n}\lambda_{\eta,\xi}E_{\eta, \xi}.$$
By the above comments about the norms of operators in $\B_n$  we conclude that 
$\|R_n-R_m\|=\|S_n-S_m\| $ for every $n,m\in \N$, and therefore
$(R_n)_{n\in \N}$ forms a Cauchy sequence in $\B(\ell_2)$ with all
elements in $\A (\mathcal T)$ and hence converges
to some operator $R\in \A (\mathcal T)$. Since for every $\xi,\eta \in A$, there is  large enough $n\in \N$ such that
$E_{\eta, \eta}R_n E_{\xi, \xi}=\lambda_{\eta, \xi}E_{\eta, \xi}$,
we conclude that $T_{\eta, \eta} R T_{\xi, \xi}=^\K E_{\eta, \eta} R E_{\xi, \xi}=\lambda_{\eta, \xi}E_{\eta, \xi}$
for all $\xi, \eta\in A$.  So $\lambda_{\eta, \xi}^{\mathcal T}(R)=\lambda_{\eta, \xi}$ for all $\xi, \eta\in A$. 
On the other hand  if $(\xi,\eta) \notin A\times A$, then 
$T_{\xi,\xi}R_n T_{\eta,\eta}\in \mathcal K (\ell_2)$ for all $n\in \N$ as 
$\{T_{\xi, \xi}: \xi \not\in A\}\cup\{E_{\xi, \xi}: \xi\in A\}$ is still almost orthogonal, and thus 
$\lambda_{\eta, \xi}^{\mathcal T}(R)= 0$.
It follows that $\Lambda^{\mathcal T}(R)=(\lambda_{\eta, \xi})_{\xi, \eta<\kappa}$ as required.
\end{proof}

The key to the proof of our main theorem is to characterize each quasi-multiplier $R$ of $\A(\mathcal T)$ based on how ``complex"  the matrix $\Lambda(R)$ is. This is captured in the following definition.

\begin{definition}\label{trivial-def} Assume that $\mathcal T=\{T_{\eta, \xi}: \xi, \eta<\kappa\}$
is a system of almost matrix units,
and $R\in \mathcal Q \mathcal M(\A(\mathcal T))$. We say
\begin{itemize}
\item $R$ is a trivial quasi-multiplier of $\A(\mathcal T)$, if $\Lambda^{\mathcal T}(R)= \lambda I_\kappa +M$, for a matrix of a compact  operator $M$ on $\B(\ell_2(\kappa))$ and $\lambda \in \mathbb C$,
\item $R$ is a $\sigma$-trivial quasi-multiplier of $\A(\mathcal T)$, if 
$\Lambda^{\mathcal T}(R)= \lambda I_\kappa +M$, for a $\kappa\times \kappa$ matrix  $M$
with at most countably many nonzero entries and some $\lambda \in \mathbb C$,
\item $R$ is a $\cc$-trivial quasi-multiplier of $\A(\mathcal T)$, if $\Lambda^{\mathcal T}(R)= \lambda I_\kappa +M$, for a $\kappa\times \kappa$ matrix  $M$
with  less than continuum
 many nonzero entries and some $\lambda \in \mathbb C$.
\end{itemize} 
\end{definition}

\begin{lemma}\label{good-matrices} Assume that $\mathcal T$ is
   a maximal system of almost 
matrix units of size $\kappa$.  
Given two quasi-multipliers $R, R'$ of $\A(\mathcal T)$, if 
 $\Lambda(R)= \Lambda(R^\prime)$, then $R=^\K  R'$. In particular,
\begin{enumerate}
\item if $\Lambda(R)$
  is a  compact $\kappa\times\kappa$-matrix, then
$R\in \mathcal A(\mathcal T)$,
\item  if $\Lambda(R)= \lambda I_\kappa$, for some $\lambda\in \mathbb C$, then
$R=^\K  \lambda 1_{\B(\ell_2)}$,
\item If $R$ is a trivial quasi-multiplier of $\A(\mathcal T)$, then 
$R\in \widetilde{\mathcal A (\mathcal T)}$.
\end{enumerate}
\end{lemma}
\begin{proof} Suppose that $R-R'$ is not compact. Then
by Lemma \ref{maximal-noncompact}, there are $\xi, \eta<\kappa$
such that $T_{\xi, \xi}(R-R')T_{\eta, \eta}$ is noncompact, and hence
by Lemma \ref{multipliers-operations} we have that
 $\lambda_{\xi, \eta}(R-R')=\lambda_{\xi, \eta}(R)-\lambda_{\xi, \eta}(R')\not=0$.

{\it(1)}  Suppose that $R\in \mathcal Q \mathcal M(\A(\mathcal T))$ is  such that 
 $(\lambda_{\xi,\eta}(R))_{\xi, \eta<\kappa}$ is a matrix of
a compact operator on $\ell_2(\kappa)$. By Lemma \ref{compact-in-at} 
we obtain $R'\in \mathcal A(\mathcal T)$ such that
 $\lambda_{\xi, \eta}(R)=\lambda_{\xi, \eta}(R')$ for
every $\xi, \eta<\kappa$.
By the first part of the lemma we conclude that $R-R'$ is compact, and therefore
$R\in A(\mathcal T)$, since $A(\mathcal T)$ includes all compact operators.

{\it(2)} Note that
 $\lambda 1_{\B(\ell_2)}$ is clearly a quasi-multiplier of $\A(\mathcal T)$, and 
 $\Lambda(\lambda 1_{\B(\ell_2)})$
 is the matrix $\lambda I_\kappa$. Now use the first part of the lemma to conclude the
 statement.

{\it(3)} 
Suppose $\Lambda^{\mathcal T}(R)= \lambda I_\kappa +M$, for a matrix of a compact  operator $M$ and $\lambda \in \mathbb C$.
By Lemma \ref{compact-in-at} there is $R^\prime \in \mathcal A(\mathcal T)$ such that 
$\Lambda(R^\prime)= M$. Then $\Lambda(R - R^\prime)= \lambda I_\kappa$ and by (2) we have 
$R-R^\prime =^\K  \lambda I_{\B(\ell_2)}$. Therefore 
$R\in \widetilde{\mathcal A (\mathcal T)}$.
\end{proof}

\begin{lemma}\label{countably-nonzero}
Assume that $\kappa$ is a  cardinal and  $\mathcal T=\{T_{\xi, \eta}: \xi, \eta<\kappa\}$
is a maximal system of almost matrix units and  $R\in \mathcal Q \mathcal M(\A(\mathcal T))$.  If $\Lambda(R)$ has at most countably nonzero entries, then it is a  $\kappa\times \kappa$ matrix
of a compact operator.
\end{lemma}
\begin{proof} 
Assume $\Lambda(R)= (\lambda_{\eta,\xi})_{\xi,\eta<\kappa}$.
By re-enumerating the $T_{\xi,\eta}$s we may assume that
if $\lambda_{\xi, \eta}\not=0$, then $\xi, \eta<\omega$. Also by
 Lemma \ref{countable-diagonalization}
we may assume that $\{T_{\xi, \xi}: \xi<\omega\}$ are pairwise orthogonal. 

By Lemma \ref{rows}, $\Lambda(R)$ is a matrix 
of a bounded linear operator on $\ell_2(\kappa)$. 
Suppose that $\Lambda(R)$ is a matrix of a noncompact operator on $\ell_2(\kappa)$. 
Aiming at a contradiction, 
we will construct a projection $P$ such that $T_{\xi,\xi} RP$ is compact for all
$\xi<\omega$ but $RP$ is noncompact. By the argument similar to Lemma \ref{maximal-noncompact}
this will give an ordinal $\xi_0<\kappa$ such that $T_{\xi_0, \xi_0} RP$ is noncompact, which by the assumption
implies that $\omega\leq \xi_0<\kappa$. Then
 $T_{\xi_0,\xi_0} R$ is also noncompact, so again
 there is $\eta_0<\kappa$ such that $T_{\xi_0,\xi_0} RT_{\eta_0, \eta_0}$ is noncompact, which means
 that $\lambda_{\xi_0, \eta_0}\not=0$ and $(\xi_0, \eta_0)\not\in \omega\times\omega$,
contradicting the hypothesis of the lemma.

To construct $P$ we will construct its range spanned by its orthonormal basis
$(v_k: k\in \omega)$. It will be enough to choose the vectors $v_k$ in such a way that
for each $\xi<\omega$ we have that $\|T_{\xi, \xi} R(v_k)\|\leq 1/2^k$ for all $k\in [\xi,\omega)$
and $\|R(v_k)\|$ does not converge to $0$ when $k\rightarrow\infty$. Then as
$P(v_k)=v_k$ and $(v_k: k\in \N)$ can be extended to an orthonormal basis of $\ell_2$, we will
obtain from Lemma \ref{compact-elementary} (1) that $T_{\xi, \xi} R P$ is compact for every $\xi<\omega$ 
and from  Lemma \ref{compact-elementary} (2) that
$RP$ is noncompact. 

Using the fact that
$\Lambda(R)$
is a matrix of a bounded linear operator (Lemma \ref{rows})
which is not compact,  
by induction on $k\in \N$ 
(using Lemma \ref{compact-elementary} (3)),  we can construct
finite pairwise disjoint  $F_k\subseteq \mathbb N=\omega$ and 
 $(a_n)_{n\in F_k}$ such that for some 
$\varepsilon>0$,
\begin{enumerate}
\item $\Sigma_{n\in F_k} |a_n|^2\leq 1$, 
\item $\|\Lambda(R)(\Sigma_{n\in F_k}a_n \chi_{\{n\}})\|=
\sqrt{\Sigma_{\xi\in \kappa}|\Sigma_{n\in F_k}a_n \lambda_{\xi, n}|^2}>\varepsilon$,
\item $|\Sigma_{n\in F_k}a_n\lambda_{\xi, n}|^2\leq \Sigma_{n\in F_k}|\lambda_{\xi, n}|^2\leq 1/2^k$
for all $\xi<k$.
\end{enumerate}
The condition (3) follows from (1) and the fact that the rows of $\Lambda(R)$
are in $\ell_2(\kappa)$ (Lemma \ref{rows}).
Using the compactness of the operators $T_{\xi, \xi}R T_{n, n}-\lambda_{\xi,n}T_{\xi, n}$, 
we find $w_n\in Im(T_{n, n})$ of norm one such that
$\|\Sigma_{\xi<\omega} \Sigma_{n\in F_k} (T_{\xi, \xi}R T_{n, n}-\lambda_{\xi,n}T_{\xi, n}) (w_n)\|<\varepsilon / 2$. 
 Putting $v_k=\Sigma_{n\in F_k} a_n w_n$  we obtain

\begin{enumerate}
\item[(a)] $\|R(v_k)\|\geq \|\Sigma_{\xi\in \omega}\Sigma_{n\in F_k} T_{\xi, \xi}R T_{n, n}(a_n w_n)\|\geq
\sqrt{\Sigma_{\xi\in \kappa}|\Sigma_{n\in F_k}a_n \lambda_{\xi, n}|^2} - \varepsilon/2>\varepsilon/2$,
\item[(b)] $\|T_{\xi, \xi} R(v_k)\|= \|\Sigma_{n\in F_k} T_{\xi, \xi} RT_{n,n} (v_k)\|\leq 1/2^k$,
for all $\xi<k$.
\end{enumerate}
As noted before this is sufficient to obtain a contradiction from the conjunction of the hypothesis 
that $\Lambda(R)$ is a matrix of a noncompact operator and
the set of its nonzero entries is countable.
\end{proof}

It was noted by the referee that the proof of above lemma can be simplified using the countable 
degree-1 saturation of the Calkin algebra (see \cite{saturation}).

\begin{corollary}\label{sigma-trivial-trivial}
Assume that $\kappa$ is a cardinal and  $\mathcal T=\{T_{\xi, \eta}: \xi, \eta<\kappa\}$
is a maximal system of almost matrix units and  $R\in \mathcal Q \mathcal M(\A(\mathcal T))$.
If $R$ is a $\sigma$-trivial quasi-multiplier of $\A(\mathcal T)$, then $R$ is a trivial quasi-multiplier of $\A(\mathcal T)$.
\end{corollary}
\begin{proof} Assume that $\Lambda^{\mathcal T}(R)$ is of the form $\lambda I_\kappa+M$
where $\lambda\in \C$ and $M$ is a matrix with countably many nonzero entries. Therefore
$\Lambda^{\mathcal T}(R-\lambda 1_{\B(\ell_2)})$ is $M$, which has at most countably many nonzero entries, which
 by Lemma \ref{countably-nonzero} means that $M$ is a matrix of a compact operator and
so $R$ is a trivial quasi-multiplier of $\A(\mathcal T)$.
\end{proof}

\section{The Cantor tree system of almost matrix units}
Let $2^{<\mathbb N}$ be the set of all  maps  
 $s:\{0, ..., n\}\rightarrow \{0,1\}$ for $n\in \N$ or $s=\emptyset$ and
  by $2^\mathbb N$  denote
the Cantor space, the space of all maps $\xi: \N\rightarrow \{0,1\}$,
 equipped with the  product topology.
For each $\xi\in 2^{\mathbb N}$ we can associate a set 
$$A_\xi = \{s\in 2^{<\mathbb N} : s\subseteq \xi\},$$
which is usually called the  ``branch through $\xi$". It is easy to see that 
$\{A_\xi :\xi \in 2^\mathbb N\}$ is an almost disjoint family of subsets of $2^{<\mathbb N}$
of size continuum.
 In this section  $\mathcal H$ denotes the  separable 
Hilbert space $ \ell_2 (2^{<\mathbb N})$.
 For each $\xi\in 2^{\mathbb N}$ define a projection $T_{\xi, \xi} \in \mathcal B(\mathcal H)$  by
$$
T_{\xi, \xi} (x) (s)= \begin{cases} x(s)   &\hbox{if } s\in A_\xi,    \\
0   & \hbox{otherwise,}
\end{cases}
$$ 
for each $x \in \mathcal H$ and $s\in 2^{<\mathbb N}$.
Then $\mathcal P_{2^\mathbb N}= \{T_{\xi,\xi} : \xi\in 2^\mathbb N\}$ is a family of almost orthogonal projections
in $\mathcal B (\mathcal H)$.  

Let $\{e_s : s\in 2^{<\N}\}$ be the canonical orthonormal basis for $\mathcal H$, i.e., 
$e_{s}(t)= 1$ if $t=s$ and $e_{s}(t)= 0$, otherwise.
For every $\xi, \eta\in 2^\N$, define a linear 
bounded operator $T_{\eta,\xi}:\mathcal H \rightarrow \mathcal H$ by
$T_{\eta,\xi}(e_{\xi|k})=e_{\eta|k}$ for every $k\in \N$ and $T_{\eta,\xi}(e_{t})=0$
if $t$ is not equal to ${\xi|k}$ for any $k\in \N$.
It is easy to see that $\mathcal T_{2^\N} = \{T_{\eta,\xi}: \xi,\eta \in 2^\mathbb N\}$ is a system of almost matrix
units (Definition \ref{almostunits}) based on $\mathcal P_{2^\N}$ (Definition \ref{maximal-system}). We will call this system of almost matrix units
``the Cantor tree system of matrix units".
 In the rest of this section the operators
$T_{\eta, \xi}$ will always refer to the members of $\mathcal T_{2^\N}$.

Recall that a family $\mathcal F$ of subsets of a Polish space 
(a separable completely metrizable  space)  is said to have the \emph{perfect  set property}, if every
uncountable element of $\mathcal F$
 has a perfect subset. In particular every uncountable
element of $\mathcal F$  must have cardinality continuum. In the following lemma
we use the fact that the family of Borel sets
 has the perfect set property (e.g., 13.6 of \cite{kechris}).

\begin{lemma}\label{analytic-set}
Assume $R\in \mathcal B (\mathcal H )$ is a quasi-multiplier of $\A(\mathcal T_{2^\N})$
 and  $U$ is a Borel subset of $\mathbb C$, then the set 
$$
B_U^R = \{(\eta, \xi)\in 2^\mathbb N\times 2^\mathbb N : \lambda_{\eta, \xi}^{\mathcal T_{2^\N}}(R) \in U \}
$$
is Borel in $2^\N\times 2^\N$. In particular,  $B^R_U$ is either countable
or of size continuum.
\end{lemma}
\begin{proof}

Let $\lambda_{\eta, \xi} = \lambda_{\eta, \xi}^{\mathcal T_{2^\N}}(R) $ for every $\xi,\eta\in 2^\N$.

\noindent{\bf Claim.} $\psi_{n}: 2^{\N}\times 2^{\N}\rightarrow \C$ 
defined by 
$$\psi_{n}(\eta, \xi)=\left\langle  R(e_{\eta|_n}), e_{\xi|_n} \right\rangle$$
is a continuous function, for every $n\in{\N}$.

\noindent{\it Proof of the Claim.} Fix $n \in \N$. For $s,t\in 2^n$
let $O_{s,t} $ denote the clopen  set $ \{(\eta, \xi)\in 2^\N \times 2^\N : s\subseteq \xi \
\& \  t\subseteq \eta\}$.
 Note that $\psi_n$ is constant on $O_{s,t}$, for every $(s,t)\in 2^n \times 2^n$.  
In fact, $\psi_n (\eta, \xi)= \langle R (e_s), e_t \rangle$ for every $(\eta, \xi) \in O_{s,t}$.
Since $2^\N \times 2^\N= \bigcup_{s,t\in 2^n} O_{s,t}$, the range of $\psi_n$ is finite, and it is
continuous, which completes the proof of the claim. 

For each $\xi, \eta\in 2^\N$, since 
$W_{\eta, \xi}= \lambda_{\eta, \xi} T_{\eta, \xi} - T_{\eta, \eta} R T_{\xi, \xi}$ 
is a compact operator in $\B(\HH)$,
we have $\lim_{n\rightarrow \infty} \|W_{\eta, \xi}(e_{\xi|_n})\|= 0$ (see Lemma \ref{compact-elementary} (2)), 
and therefore $$\lim_{n\rightarrow\infty} |\langle W_{\eta, \xi}(e_{\xi|_n}), e_{\eta|_n}\rangle| = 0.$$ 
This means  that
$$ \left\langle T_{\eta, \eta} RT_{\xi, \xi}(e_{\xi|_n}), e_{\eta|_n}\right\rangle \rightarrow
 \left\langle \lambda_{\eta, \xi} T_{\eta, \xi}(e_{\xi|_n}), e_{\eta|_n} \right\rangle = 
 \lambda_{\eta, \xi}\left\langle e_{\eta|_n} , e_{\eta|_n} \right\rangle = \lambda_{\eta, \xi}. $$

Thus, for each $\xi, \eta\in 2^\N$
$$\psi_n (\eta, \xi)=\left\langle  R(e_{\eta|_n}), e_{\xi|_n} \right\rangle= 
  \left\langle T_{\eta, \eta} RT_{\xi, \xi} (e_{\xi|_n}), e_{\eta|_n}\right\rangle,
 $$
  converges to $\lambda_{\eta, \xi}$. So the map
 $\psi: 2^\N \times2^\N \rightarrow \mathbb C$ given by $\psi(\eta, \xi)=\lambda_{\eta, \xi}$
is the pointwise limit of continuous functions $\psi_n$ for $n\in \N$, hence it is Borel
(Ex. 11.2 (i)  \cite{kechris}), which means
  $B_U^R = \psi^{-1}[U]$  is Borel.

\end{proof}

\begin{corollary}\label{c-trivial-sigma-trivial} If $R$ is a $\cc$-trivial quasi-multiplier of $\A(\mathcal T_{2^\N})$, then
$R$ is a $\sigma$-trivial quasi-multiplier of $\A(\mathcal T_{2^\N})$.
\end{corollary}
\begin{proof} Suppose that $\Lambda^{\mathcal T_{2^\N}}(R)=\lambda I_{2^\N}+M$ where $\lambda\in \C$
and  $M$ has 
less then continuum nonzero entries. 
Note that $\Lambda^{\mathcal T_{2^\N}}(R-\lambda 1_{\B(\HH)})=
\lambda I_{2^\N}+M-\lambda I_{2^\N}=M$ by Lemma \ref{multipliers-operations}.
So 
$B^{R-\lambda 1_{\B(\HH)}}_{\C\setminus\{0\}}$ is  
a Borel subset of 
$2^{\N}\times 2^\N$ (Lemma  \ref{analytic-set}), and
of cardinality less then $\cc$.  Therefore the perfect set property for Borel sets,
implies that $B^{R-\lambda 1_{\B(\HH)}}_{\C\setminus\{0\}}$ is countable, 
 so $M$ has at most countably many
nonzero entries, which means that $R$ is $\sigma$-trivial.
\end{proof}

\section{Pairing systems of almost matrix units}
In this section we introduce a method of eliminating nontrivial quasi-multipliers of $\mathcal A(\mathcal T)$
by pairing the elements of $\mathcal T$ into a new system of almost matrix units. 
\begin{definition}\label{pairing-def} Let $X$ be a set which is partitioned into two
subsets of the same cardinality, $X=Y\cup (X\setminus Y)$ 
and suppose that $\rho: Y\rightarrow (X\setminus Y)$ is a bijection.
Suppose that $\mathcal T=\{T_{\eta, \xi}: \xi, \eta\in X\}$ is a
system of almost matrix units in $\B(\ell_2)$. We say 
$\mathcal U=\{U_{\eta, \xi}: \xi, \eta\in Y\}$ is a pairing of $\mathcal T$ along $\rho$
if and only if for every $\xi, \eta\in Y$ the following holds:
$$U_{\eta, \xi}=^\K T_{\eta, \xi}+T_{\rho(\eta), \rho(\xi)}.$$

\end{definition}

\begin{proposition}\label{pairing-existence} Let $X$, $Y$, $\rho$ and $\mathcal T$ be as above.
 Then any pairing $\mathcal U$
of $\mathcal T$ along $\rho$ is a system of almost matrix units. If $\mathcal T$ is maximal, then $\mathcal U$ 
  is also  maximal.

\end{proposition}
\begin{proof} 
Let $\mathcal U=\{U_{\eta, \xi}: \xi, \eta\in Y\}$ be a pairing of $\mathcal T$ along $\rho$. Then 
for every $\xi, \eta\in Y$ we have
$$U_{\eta, \xi}=^\K T_{\eta, \xi}+T_{\rho(\eta), \rho(\xi)}.$$ 

We check that $\mathcal U =\{U_{\eta, \xi}: \xi, \eta\in Y\}$ is a system of almost
matrix units:  $(U_{\eta,\xi})^*=^\K (T_{\eta, \xi}+T_{\rho(\eta), \rho(\xi)})^*
=^\K T_{\xi, \eta}+T_{\rho(\xi), \rho(\eta)}=^\K U_{\xi, \eta}$ for all $\xi, \eta\in Y$.
For all $\alpha, \beta, \xi, \eta\in Y$, since $Y\cap \rho[Y]=\emptyset$ and
 $\rho$ is a bijection, a straightforward calculation show that
\begin{align*}
U_{\beta, \alpha}U_{\eta, \xi}&=^\K  \delta_{\alpha,\eta}U_{\beta, \xi}.
\end{align*}

Now suppose that
 $\mathcal T=\{T_{\eta, \xi}: \xi, \eta\in X\}$ 
 is a maximal system of almost matrix units, that is,
there is a maximal family $\{P_\xi: \xi\in X\}$ of almost orthogonal 
projections (see Definition \ref{maximal-system}) such that $T_{\xi, \xi}=^\K P_\xi$
for each $\xi\in X$.  We will show that
$\mathcal U$ is also a maximal system of almost matrix units. We need to produce 
a maximal family $\mathcal Q= \{Q_\xi: \xi \in Y\}$   of almost orthogonal projections, such that $\mathcal U$ is based on $\mathcal Q$.

Using Lemma \ref{countable-diagonalization} for 
each pair $s=\{\xi, \rho(\xi)\}$, separately for every $\xi\in Y$ find
orthogonal projections $P_\xi^s, P_{\rho(\xi)}^s\in \B(\ell_2)$ such that 
$P^s_\xi=^\K P_\xi$ and $P_{\rho(\xi)}^s=^\K P_{\rho(\xi)}$
and $P_\xi^sP_{\rho(\xi)}^s=0$. 
For each $\xi\in Y$ define $Q_\xi=P_\xi^s+P_{\rho(\xi)}^s$ for $s=\{\xi, \rho(\xi)\}$,
which is a projection as it is the sum of two orthogonal projections and moreover
$$Q_\xi=^\K T_{\xi,\xi}+T_{\rho(\xi), \rho(\xi)}=^\K U_{\xi, \xi}.\leqno (*)$$
It remains to prove that $\mathcal Q$ is a maximal family of
almost orthogonal projections. 
Suppose that $P$ is a projection in $\mathcal B (\mathcal \ell_2)$.  By the maximality of 
$\{P_\xi: \xi\in X\}$,
there is $\alpha\in X$ such that $P_\alpha P$ is not a compact operator. 
Let $\xi\in Y$ be such that $\alpha\in \{\xi, \rho(\xi)\}$, so we have 
$$T_{\alpha, \alpha}U_{\xi, \xi}=^\K T_{\alpha, \alpha}(T_{\xi, \xi}+T_{\rho(\xi), \rho(\xi)})=^\K T_{\alpha, \alpha},$$
by Definition  \ref{almostunits} (2), as the domain and the range of $\rho$ are disjoint.
Therefore $T_{\alpha, \alpha} U_{\xi, \xi} 
P =^\K  T_{\alpha, \alpha} P=^\K P_\alpha P$. Thus $T_{\alpha, \alpha} U_{\xi, \xi} P$ and consequently  by $(*)$
$U_{\xi, \xi}P=^\K Q_\xi P$ are noncompact, which shows that $\mathcal U$ is maximal as well.
\end{proof}

\begin{lemma}\label{pairing} Suppose $X, Y$ and $\rho$ are as in Definition \ref{pairing-def}.
Let  $\mathcal T=\{T_{\eta, \xi}: \xi, \eta\in X\}$ be a
system of almost matrix units and $\mathcal U$ be a pairing
of $\mathcal T$ along $\rho$. Suppose that $R\in \B(\ell_2)$ is a quasi-multiplier
for $\A(\mathcal U)$. Then $R$ is a quasi-multiplier of $\A(\mathcal T)$ and
$$\lambda_{\eta, \xi}^{\mathcal U}(R)=\lambda_{\eta, \xi}^{\mathcal T}(R)
=\lambda_{\rho(\eta), \rho(\xi)}^{\mathcal T}(R),$$
$$\lambda_{\eta, \rho(\xi)}^{\mathcal T}(R)=\lambda_{\rho(\eta), \xi}^{\mathcal T}(R)=0$$
for each $\xi, \eta\in Y$.
\end{lemma}
\begin{proof}
We have $T_{\xi, \xi}U_{\xi, \xi}=^\K T_{\xi, \xi}(T_{\xi, \xi}+ T_{\rho(\xi), \rho(\xi)})=^\K T_{\xi, \xi}$,
as $Y\cap \rho[Y]=\emptyset$ and  by the almost
matrix units relations. Similarly  $U_{\xi, \xi}T_{\xi, \xi}=^\K T_{\xi, \xi}$. Then
\begin{align*}
T_{\eta, \eta}RT_{\xi, \xi}&=^\K \lambda_{\eta, \xi}^{\mathcal U}(R)T_{\eta, \xi},
\end{align*}
again since $Y\cap \rho[Y]=\emptyset$.

Similarly $U_{\xi, \xi}T_{\rho(\xi), \rho(\xi)}=^\K T_{\rho(\xi), \rho(\xi)}$
and  $T_{\rho(\xi), \rho(\xi)}U_{\xi, \xi}=^\K T_{\rho(\xi), \rho(\xi)}$ and so\break
$T_{\rho(\eta), \rho(\eta)}RT_{\rho(\xi), \rho(\xi)}
=^\K \lambda_{\eta, \xi}^{\mathcal U}(R)T_{\rho(\eta), \rho(\xi)}$.

To prove the  second part of the lemma note that
$\lambda_{\eta, \xi}^{\mathcal U}(R)(T_{\eta, \xi}+T_{\rho(\eta), \rho(\xi)})
=^\K \lambda_{\eta, \xi}^{\mathcal U}(R)U_{\eta, \xi}=^\K U_{\eta, \eta}RU_{\xi, \xi}=^\K
(T_{\eta, \eta}+T_{\rho(\eta), \rho(\eta)})R(T_{\xi, \xi}+T_{\rho(\xi), \rho(\xi)})
=^\K 
\lambda_{\eta, \xi}^{\mathcal T}(R)T_{\eta, \xi}+
\lambda_{\eta, \rho(\xi)}^{\mathcal T}(R)T_{\eta, \rho(\xi)}+
\lambda_{\rho(\eta), \xi}^{\mathcal T}(R)T_{\rho(\eta), \xi}+
\lambda_{\rho(\eta), \rho(\xi)}^{\mathcal T}(R)T_{\rho(\eta), \rho(\xi)}$.
Multiplying the above equalities by $T_{\eta, \eta}$ from the left and $T_{\rho(\xi), \rho(\xi)}$ from the right,
using \ref{almostunits} (2) and the fact that $\rho(\eta)\not=\eta\not=\rho(\xi)\not=\xi$ we obtain that 
$0=^\K \lambda_{\eta, \rho(\xi)}^{\mathcal T}(R)T_{\eta, \rho(\xi)}$. Since
$T_{\eta, \rho(\xi)}$ is noncompact, it follows that $\lambda_{\eta, \rho(\xi)}^{\mathcal T}(R)=0$.
We obtain $\lambda_{\rho(\eta), \xi}^{\mathcal T}(R)=0$ in a similar way
multiplying the above equalities by $T_{\xi, \xi}$ from the right and $T_{\rho(\eta), \rho(\eta)}$ from the left.
\end{proof}

\begin{lemma}\label{pairing-c-trivial} Suppose $X, Y$ and $\rho$ are as in Definition \ref{pairing-def}.
Let  $\mathcal T=\{T_{\eta, \xi}: \xi, \eta\in X\}$ be a
system of almost matrix units and $\mathcal U$ be a pairing
of $\mathcal T$ along $\rho$. If $R\in \B(\ell_2)$ is a $\sigma$-trivial  ($\cc$-trivial) quasi-multiplier
for $\A(\mathcal U)$, then $R$ is a  $\sigma$-trivial  ($\cc$-trivial) quasi-multiplier of $\A(\mathcal T)$. 
\end{lemma}
\begin{proof} The $X\times X$ matrix $\Lambda^{\mathcal T}(R)$ consists of
four blocks $Y\times Y$, $(X\setminus Y)\times Y$, $Y\times (X\setminus Y)$ 
and $(X\setminus Y)\times (X\setminus Y)$. Lemma \ref{pairing} implies that
the $Y\times Y$-block is the matrix $\Lambda^{\mathcal U}(R)$, that
$(X\setminus Y)\times (X\setminus Y)$-block is a copy of the $Y\times Y$-block and
the remaining  blocks have only zero entries. This clearly implies the lemma.
\end{proof}

\begin{lemma}\label{killer-pairing} Suppose that $\mathcal T=\{T_{\eta, \xi}: \xi, \eta \in X\}$ is a system
of almost matrix units where $X$ is of size  continuum. 
Then there are $Y\subseteq X$ and a bijection $\rho: Y\rightarrow (X\setminus Y)$ 
 such that for every  pairing $\mathcal U$ of $\mathcal T$ along $\rho$, 
whenever $R$ is a quasi-multiplier of $\A(\mathcal U)$, then $R$ is a $\cc$-trivial quasi-multiplier of
$\A(\mathcal T)$.
\end{lemma}
\begin{proof} We may assume that $X=\cc$.
Let $(R_\xi)_{\xi<\cc}$ be an enumeration (with possible repetitions) of
all quasi-multipliers of $\A(\mathcal T)$ which are not $\cc$-trivial. 
By induction on $\alpha<\cc$ we construct distinct $\beta_\alpha^i, \gamma_\alpha^i\in \cc$
for $i=1, 2, 3$ 
such that $\{\beta_\alpha^i, \gamma_\alpha^i: i\in \{1, 2, 3\}\}$ has six distinct elements for each $\alpha<\cc$
 and such that
either
\begin{enumerate}

\item $\lambda^{\mathcal T}_{\beta_\alpha^1, \beta_\alpha^2}(R_\alpha)\not=0$
and $\lambda^{\mathcal T}_{\gamma_\alpha^1, \gamma_\alpha^2}(R_\alpha)=0$, or
\item $\lambda^{\mathcal T}_{\beta_\alpha^3, \beta^3_\alpha}(R_\alpha)\not=
\lambda^{\mathcal T}_{\gamma^3_\alpha, \gamma^3_\alpha}(R_\alpha)$, 
\end{enumerate}
and moreover $\{\beta_\alpha^i, \gamma_\alpha^i: i\in \{1, 2, 3\}, \alpha<\cc\}=\cc$.

At stage $\alpha<\cc$ consider the set 
$A_\alpha=\{\beta_\delta^i, \gamma_\delta^i: i\in \{1, 2, 3\}, \delta<\alpha\}$.
Before defining $\{\beta_\alpha^i, \gamma_\alpha^i: i\in \{1, 2, 3\}$
we will identify  the reason why a quasi-multiplier $R_\alpha$ is not
a $\cc$-trivial quasi-multiplier of $\A(\mathcal T)$. If it is because  $\Lambda^{\mathcal T}(R_\alpha)$
has continuum nonzero entries off the diagonal, then we find such an entry 
$\lambda_{\eta, \xi}^{\mathcal T}(R_\alpha)$ with distinct $\xi, \eta\not \in A_\alpha$. This can be
achieved because by Lemma \ref{rows} the cardinality of the set of all  nonzero entries
$\lambda_{\xi, \eta}$ with  $\xi, \eta\in A_\alpha$ is less than $\cc$ and we 
have assumed that $\Lambda^{\mathcal T}(R_\alpha)$
has continuum nonzero entries off the diagonal.  Now find distinct $\xi',\eta'\not \in 
A_\alpha\cup\{\xi, \eta\}$
 so that $\lambda_{\eta', \xi'}^{\mathcal T}(R_\alpha)=0$.
This can be achieved again by Lemma \ref{rows}.
Put $\beta_\alpha^1=\xi, \gamma_\alpha^1=\xi', \beta_\alpha^2=\eta, \gamma_\alpha^2=\eta'$
so  (1) holds.
Now take 
$\beta_\alpha^3, \gamma_\alpha^3$ be the first
two elements of the set 
$\cc\setminus (A_\alpha\cup\{\beta_\alpha^1, \beta_\alpha^2, \gamma_\alpha^1, \gamma_\alpha^2\})$.

Otherwise if $\Lambda^{\mathcal T}(R_\alpha)$
has less then continuum nonzero entries off the diagonal but is not a $\cc$-trivial quasi-multiplier of
$\A(\mathcal T)$. Then it must be the case that $\Lambda^{\mathcal T}(R_\alpha)$ has two different entries
$\lambda_{\xi, \xi}^{\mathcal T}(R_\alpha)\not=\lambda_{\eta, \eta}^{\mathcal T}(R_\alpha)$
on the diagonal such that $\xi, \eta\not\in A_\alpha$, since $A_\alpha$
  has cardinality less than continuum.  
So we put $\beta_\alpha^3=\xi$, $\gamma_\alpha^3=\eta$ so that (2) holds.

In this case  put 
$\beta_\alpha^1, \gamma_\alpha^1, \beta_\alpha^2, \gamma_\alpha^2$ to be the first
four elements of the set 
$\cc\setminus (A_\alpha\cup\{\beta_\alpha^3, , \gamma_\alpha^3\}))$.
The choice of $\beta_\alpha^3, \gamma_\alpha^3$ in the first case and $\beta_\alpha^1, \gamma_\alpha^1, \beta_\alpha^2, \gamma_\alpha^2$ in the second case guarantees that
$\{\beta_\alpha^i, \gamma_\alpha^i: i\in \{1, 2, 3\}, \alpha<\cc\}=\cc$, which completes the
inductive construction.

We put $Y=\{\beta_\alpha^i: i\in \{1,2, 3\}, \alpha<\cc\}$ and 
we define $\rho: Y\rightarrow (X\setminus Y)$ by $\rho(\beta_\alpha^i)=\gamma_\alpha^i$.
Let $\mathcal U=\{U_{\eta, \xi}: \xi, \eta\in Y\}$ be a pairing of $\mathcal T$ along $\rho$.

Suppose that $R$ is a quasi-multiplier of $\A(\mathcal T)$ which is not 
$\cc$-trivial, so $R=R_\alpha$ for some $\alpha<\cc$. We will show that
$R$ is not a quasi-multiplier of $\A(\mathcal U)$, which will prove the required  property 
of $\mathcal U$.

If (1) holds, then $\lambda^{\mathcal T}_{\beta^1_\alpha, \beta^2_\alpha}(R)\not=
\lambda^{\mathcal T}_{\rho(\beta^1_\alpha), \rho(\beta^2_\alpha)}(R)$ as
$\rho(\beta^i_\alpha)=\gamma^i_\alpha$ for $i=1, 2$, but this contradicts Lemma \ref{pairing}.

If (2) holds, then
$\lambda^{\mathcal T}_{\beta^3_\alpha, \beta^3_\alpha}(R)\not=
\lambda^{\mathcal T}_{\rho(\beta^3_\alpha), \rho(\beta^3_\alpha)}(R)$ as
$\rho(\beta^3_\alpha)=\gamma^3_\alpha$, but this contradicts Lemma \ref{pairing}.
This shows that $R$ is not a quasi-multiplier of $\A(\mathcal U)$ and completes the proof of the lemma.
\end{proof}

\begin{lemma}\label{upward-triviality}
Suppose that $\mathcal T=\{T_{\eta, \xi}: \xi, \eta \in X\}$ is a system
of almost matrix units and $\rho: Y\rightarrow (X\setminus Y)$ is a bijection
where $Y\subseteq X$ and that  $\mathcal U$ is a pairing
of $\mathcal T$ along $\rho$. If 
$R\in \B(\ell_2)$ is a quasi-multiplier of $\A(\mathcal U)$ 
which  is a $\sigma$-trivial quasi-multiplier of $\A(\mathcal T_1)$, where $\mathcal T_1=\{T_{\eta, \xi}: \xi, \eta \in Y'\}$,
for any $Y'$ satisfying $Y\subseteq Y'\subseteq X$,
then  $R$ is a $\sigma$-trivial quasi-multiplier of $\A(\mathcal U)$.
 \end{lemma}
\begin{proof} 
Let $\lambda\in \C$ be such that $\Lambda^{\mathcal T_1}(R-\lambda 1_{\B(\ell_2)})$ 
is a matrix with countably many nonzero entries.
By Lemma \ref{pairing}  there are only 
countably many nonzero entries of $\Lambda^{\mathcal U}(R-\lambda 1_{\B(\ell_2)})$,
because they are all equal to some entries  of
$\Lambda^{\mathcal T_1}(R-\lambda 1_{\B(\ell_2)})$, so
$R$  is a $\sigma$-trivial quasi-multiplier of $\A(\mathcal U)$.

\end{proof}

\section{The final construction}

The construction of the $C^*$-algebra indicated in the title of this paper and described in the
introduction starts with the Cantor tree system of almost matrix units $\mathcal T_{2^\N}$ and follows
the scheme:

$$\mathcal T_{2^\N}\xrightarrow{\rm \ extending\ }
\{T_{\xi, \eta}: \xi, \eta\in 2^\N\cup X\}\xrightarrow{\rm \ pairing\  with\ Y\subseteq  2^\N\ }\mathcal U
\xrightarrow{\rm \ pairing, \ Lemma\ 5.5.\ }\mathcal S$$
\vskip 1cm
\begin{theorem}\label{main} There is a type I $C^*$-subalgebra $\mathcal A$ of $\B(\ell_2)$
containing the ideal of compact operators $\K(\ell_2)$ such that
$\A/\K(\ell_2)$ is $*$-isomorphic to the algebra $\K(\ell_2(\cc))$ of all compact operators
 on the Hilbert space of density continuum and the algebra $\mathcal M(\A)$
of multipliers of $\A$ is equal to the unitization $\tilde{\A}$ of $\A$.
\end{theorem}
\begin{proof} We work with $\ell_2(2^{<\N})$ instead of $\ell_2$, as $2^{<\N}$ is countable.
Start with the Cantor tree system of  almost matrix units $\mathcal T_{2^\N}$ of Section 4.
Extend it to a maximal system of almost matrix units $\{T_{\xi, \eta}: \xi, \eta\in 2^\N\cup X\}$
for some set $X$, by Lemma \ref{extension-to-maximal}.  It is clear that $X$ has cardinality
not bigger than continuum. Let $Y\subseteq  2^\N$ be such that both $Y$ and $2^\N\setminus Y$
have cardinality $\cc$. Fix a bijection $\rho: Y\rightarrow (2^\N\setminus Y)\cup X$.
Now let $\mathcal U$ be a  pairing of  $\{T_{\xi, \eta}: \xi, \eta\in 2^\N\cup X\}$
along $\rho$. 
Finally apply the pairing from Lemma \ref{killer-pairing} 
(for $\mathcal U$ instead of $\mathcal T$), to obtain  a system $\mathcal S$ of almost
matrix units with the special properties mentioned in the Lemma \ref{killer-pairing}.
We claim that $\mathcal A(\mathcal S)$ is the desired $C^*$-algebra.

So suppose that $R$ is in the multiplier algebra $\mathcal M(\mathcal A(\mathcal S))$ 
of $\mathcal A(\mathcal S)$. Then
$R$ is a quasi-multiplier of $\A(\mathcal S)$. Lemma \ref{killer-pairing} implies that 
$R$ is a $\cc$-trivial quasi-multiplier of $\A(\mathcal U)$ and so Lemma \ref{pairing-c-trivial}
implies that $R$ is a $\cc$-trivial quasi-multiplier of $\A(\{T_{\xi, \eta}: \xi, \eta\in 2^\N\cup X\})$
and hence for $\A(\mathcal T_{2^\N})$. This however
implies that $R$ is a $\sigma$-trivial quasi-multiplier of $\A(\mathcal T_{2^\N})$ by Corollary
\ref{c-trivial-sigma-trivial}.  By Lemma \ref{upward-triviality} (for $Y'=2^\N$) 
the operator $R$ is  a $\sigma$-trivial quasi-multiplier of $\A(\mathcal U)$ and again  by Lemma \ref{upward-triviality} 
(for $\mathcal U$ as $\mathcal T=\mathcal T_1$ and $Y' = Y$) 
it is $\sigma$-trivial for $\A(\mathcal S)$. However $\{T_{\xi, \eta}: \xi, \eta\in 2^\N\cup X\}$
was a maximal system of almost matrix units, so by the last part of
Lemma \ref{pairing-existence}, the system $\mathcal U$ and hence $\mathcal S$ are maximal
systems of almost matrix units. The  maximality of $\mathcal S$ together with the fact
that $R$ is a $\sigma$-trivial quasi-multiplier of $\A(\mathcal S)$ implies that 
$R$ is a trivial quasi-multiplier of $\A(\mathcal S)$ (Corollary \ref{sigma-trivial-trivial}).
Trivial quasi-multipliers of maximal systems 
of almost matrix units belong to the unitizations of the algebra generated by them and the
compact operators, by  Lemma \ref{good-matrices} (3). Therefore $R$ belongs
to the unitization of $\mathcal A(\mathcal S)$, as required.
 
\end{proof}

\bibliographystyle{amsplain}

\end{document}